\documentclass[11pt,oneside]{amsart}

\usepackage{latexsym}
\usepackage{amsmath}
\usepackage{amssymb}
\usepackage{amsthm}
\usepackage{wasysym}
\usepackage{cite}
\usepackage{graphicx}
\usepackage{color}
\usepackage{enumitem}
\usepackage{bm}
\usepackage{xspace}
\setlist[description]{font=\normalfont\itshape\textbullet\space}
\usepackage[all]{xy}
\usepackage{hyperref}

\renewcommand{\paragraph}[1]{\vspace{6pt} \noindent \textbf{#1}\xspace}

\numberwithin{equation}{section}
\newtheorem{theorem}{Theorem}[section]

\newtheorem{lemma}[theorem]{Lemma}

\newtheorem{proposition}[theorem]{Proposition}
\newtheorem{claim}[theorem]{Claim}

\newcommand{\IP}{\mathrm{I}}
\newcommand{\TIP}{\mathrm{TI}}
\newcommand{\AP}{\mathrm{Ab}}
  {\end{tabular}\par\medskip}

\theoremstyle{definition}

\newtheorem{definition}[theorem]{Definition}
\newtheorem{example}[theorem]{Example}

\newcommand{\F}{\mathbb{F}}

\newcommand{\Z}{\mathbb{Z}}

\newcommand{\N}{\mathbb{N}}

\renewcommand{\span}{\mathrm{span}}

\newcommand{\tr}[1]{#1^{\mathrm{t}}}

\newcommand{\M}{\mathrm{M}}

\newcommand{\Gr}{\mathrm{Gr}}

\newcommand{\spa}[1]{\mathcal{#1}}

\newcommand{\cA}{\spa{A}}
\newcommand{\cB}{\spa{B}}
\newcommand{\cC}{\spa{C}}
\newcommand{\cG}{\spa{G}}

\newcommand{\qbinom}[3]{\genfrac{[}{]}{0pt}{}{#1}{#2}_{#3}}

\DeclareMathOperator{\id}{id}

\renewcommand{\le}{\leqslant}

\newcommand{\too}%
{\xrightarrow{\text{\raisebox{-3pt}{$\sim$}}\,}}

\hypersetup{
    pdftitle={A $q$-analogue of independence polynomials},
    pdfauthor={Youming Qiao},
    bookmarksnumbered=true,     
    bookmarksopen=true,         
    bookmarksopenlevel=1,       
    colorlinks=true,            
    pdfstartview=Fit,           
    pdfpagemode=UseOutlines,    
    pdfpagelayout=OneColumn,
    pdfstartview=FitH
}

\title[A $q$-analogue of graph independence polynomials]{
A \lowercase{$q$}-analogue of graph independence polynomials
with a group-theoretic interpretation}

\author{
Youming Qiao
}

\address{
Centre for Quantum Software and Information, University of Technology Sydney.
}

\email{
jimmyqiao86@gmail.com
}

\date{ \today
}
\subjclass[2020]{}
\thanks{}

\begin{document}


\maketitle

\begin{abstract}
We define totally-isotropic space polynomials of alternating matrix spaces over finite fields, 
by analogy with independence polynomials of graphs. Our main result shows that totally-isotropic 
polynomials of \emph{graphical} alternating matrix spaces give rise to a natural $q$-analogue of graph independence polynomials. 

For $p$-groups of class $2$ and exponent $p$, this family of polynomials over fields of order $p$ can be naturally interpreted as enumerating their abelian subgroups containing the commutator subgroup according to the orders. With this interpretation, our main result has implications to graphical groups over $\F_p$ where $p$ is an odd prime, in the same spirit as the results in (\emph{Bull. Lond. Math. Soc.}, 2022) by Rossmann, who studied enumerating conjugacy classes of graphical groups over finite fields.
\end{abstract}

%


\section{Introduction}


\subsection{Background}

\subsubsection{Graph independence polynomials.} Let $G=(S, E)$ be a finite, simple, and undirected graph, where $S$ is a finite vertex set, and $E\subseteq \binom{S}{2}$ is the edge set. We say that $W\subseteq S$ is an \emph{independent set} of $G$, if $W$ does not contain any edge in $E$. The maximum independent set size of $G$ is denoted by $\alpha(G)$. 

Let $x$ be a variable. For $n\in \N$, $[n]:=\{1, 2, \dots, n\}$. The independence polynomial of $G$ in $x$ is the enumeration polynomial of independent sets of size $i$. That is, for $i\in\N$, let 
$c_i(G)$ be the number of independent sets of size $i$ in $G$. The \emph{independence 
polynomial} of $G$ is 
$$\IP(G, x):=1+\sum_{i\in[\alpha(G)]}c_i(G)\cdot x^i\in \Z[x].$$

Independence polynomials of graphs were first introduced by Gutman and Harary \cite{GH83} and have received considerable attention in combinatorics and physics. We refer the readers to \cite{LM05,CS07,Bar16} for more research on this. 

A basic property of graph independence polynomials is as follows. Let $G_1=(S_1, E_1)$ and $G_2=(S_2, E_2)$ be two graphs with disjoint vertex sets. Their union is then $G=(S, E)$, where $S=S_1\cup S_2$ and $E=E_1\cup E_2$. It is easy to observe (see e.g. \cite{LM05})
\begin{equation}\label{eq:graph_union}
\IP(G, x)=\IP(G_1, x)\cdot \IP(G_2, x).
\end{equation}

\subsubsection{Totally-isotropic spaces of alternating matrix spaces.} Let $\F_q^n$ be the linear space of length-$n$ column vectors over the finite field $\F_q$ of order $q$. We use $\M(n, q)$ to denote the linear space of $n\times n$ matrices over $\F_q$. 

A matrix $B\in\M(n,q)$ is \emph{alternating} if for any $u\in \F_q^n$, $\tr{u}Bu=0$. We use $\Lambda(n, q)$ to denote the linear space of $n\times n$ alternating matrices over $\F_q$. A subspace $\cB$ of $\Lambda(n, q)$ is called an \emph{alternating matrix space}, denoted by $\cB\leq\Lambda(n, q)$. 

\begin{definition}
	Let $\cB\leq\Lambda(n, q)$ be an alternating matrix 
	space. We say that $U\leq \F_q^n$ is a \emph{totally-isotropic space} of $\cB$, if for every $u_1, 
	u_2\in U$, and every $B\in \cB$, $\tr{u_1}Bu_2=0$. 
	
	The maximum totally-isotropic 
	space dimension of $\cB$ is denoted by $\alpha(\cB)$. We use $c_i(\cB)$ to denote the 
	number of $i$-dimensional totally-isotropic subspaces of $\cB$. 
\end{definition}

\subsubsection{Some previous works on totally-isotropic spaces.} Let $p$ be a prime. Recall that a group $P$ is a \emph{$p$-group of class $2$ and exponent $p$}, if $|P|$ is a power of $p$, every $g\in P$ satisfies $g^p=\id$ (the exponent condition), and the commutator subgroup $[P,P]$ is contained in the centre $\mathrm{Z}(P)$ (the nilpotency class $2$ condition). When $p$ is odd, 
$p$-groups of class $2$ and exponent $p$ give rise to alternating bilinear maps, and therefore alternating matrix spaces, via Baer's correspondence \cite{Bae38}. Via this correspondence, totally-isotropic spaces of alternating matrix spaces naturally correspond to certain abelian subgroups of $p$-groups of class $2$ and exponent $p$. As a result, several classical works in group theory, such as \cite{Alp65,Ols78,BGH87}, studied totally-isotropic spaces with group-theoretic implications, as will be explained in more detail in Section~\ref{subsec:group}.

Some recent works studied totally-isotropic spaces of alternating matrix spaces in analogy with independent sets of graphs \cite{BCG+21,Qia20_extremal}. This is due to a natural procedure of associating alternating matrix spaces with graphs, by sending an edge $\{i, j\}\in \binom{[n]}{2}$, $i<j$, to an $n\times n$ elementary alternating matrix $A_{i, j}$ where the $(i, j)$th entry is $1$, the $(j, i)$th entry is $-1$, and the rest being $0$. From a graph $G=([n], E)$ and a field $\F$, this construction produces a \emph{graphical (alternating) matrix space} $\cB_G\leq \Lambda(n, \F)$. It was first used by Tutte \cite{Tut47} and Lov\'asz \cite{Lov79} in the context of graph perfect matchings. 

The correspondence between independent sets and totally-isotropic spaces was first studied in \cite{BCG+21}. For example, one result in \cite{BCG+21} shows that the independence number $\alpha(G)$ is equal to $\alpha(\cB_G)$, the totally-isotropic number of its corresponding graphical matrix space (regardless of the underlying field). This connection was extended to hypergraphs and spaces of alternating multilinear forms \cite{Qia20_extremal}.

More connections between graphs and matrix spaces can be found in \cite{LQWWZ}.

\subsection{Our results: totally-isotropic space polynomials of alternating matrix spaces}

\subsubsection{Enumerating totally-isotropic spaces of graphical matrix spaces.}
Recall that given $\{i, j\}\in 
\binom{[n]}{2}$, $i<j$, the $n\times n$ elementary alternating matrix $A_{i,j}$ is 
the 
matrix with the $(i,j)$th entry being $1$, $(j, i)$th entry being 
$-1$, and other entries being $0$. 

Given a graph $G=([n], E)$ where $E\subseteq\binom{[n]}{2}$, 
let $\cG_q:=\span\{A_{i,j} \mid \{i, j\}\in E\}\leq \Lambda(n, q)$. 

Our first main result is concerned with enumerating totally-isotropic spaces of graphical matrix spaces. The proof of the following theorem is in Section~\ref{sec:main}.
\begin{theorem}\label{thm:main_enum}
	For any $1\leq i\leq \alpha(G)$, there exists $c_i(y)\in\Z[y]$, such that $c_i(q)=c_i(\cG_q)$ for a prime power $q$, and $c_i(1)=c_i(G)$. 
\end{theorem}
Theorem~\ref{thm:main_enum} shows that there exist polynomials $c_i(y)$ that capture simultaneously the cases of the number of size-$i$ independent sets for $G$ with $y=1$, and 
the number of $i$-dimensional totally-isotropic spaces for $\cG_q$ with $y=q$ a prime power.

\subsubsection{Totally-isotropic space polynomials of alternating matrix spaces.} Based on Theorem~\ref{thm:main_enum}, we define totally-isotropic space polynomials for alternating matrix spaces over finite fields, following graph independence polynomials.
As in the graph case, such polynomials are enumeration polynomials of $i$-dimensional totally-isotropic spaces of alternating matrix spaces over $\F_q$. To define such polynomials, there is one twist in the choice of bases of polynomial rings. In the graph setting, the basis is naturally $\{x^i\mid i\in\N\}\subseteq \Z[x]$ and Equation~\ref{eq:graph_union} follows easily. In the alternating matrix space setting, we introduce the following. 

Let $x$ and $y$ be variables. 
In the alternating matrix space setting, we make use of the following basis of $\Z[y][x]$: for $d=0$, set $x_y^d:=1$. For $d\geq 1$, we set
\begin{equation}\label{eq:xyd}
x_y^d:=x\cdot (x-(y-1))\cdot \ldots\cdot (x-(y^{d-1}-1)).
\end{equation}
Note that $x^d$ is the leading term in $x_y^d$, and its coefficient is $1$. Therefore, in the $\Z[y]$-module $\Z[y][x]$, the change-of-basis matrix from $\{x^d\mid d\in \N\}$ to $\{x_y^d\mid d\in\N\}$ is unitriangular. It follows that $\{x_y^d\mid d\in\N\}$ is a basis of $\Z[y][x]$ as a $\Z[y]$-module.

We shall use $x_y^d$ with $y$ assigned as $1$ or a prime power $q$. Note that $x_1^d=x^d$. 

We now define the following.
\begin{definition}
Let $\cB\leq\Lambda(n, q)$ be an alternating matrix space. The \emph{totally-isotropic space polynomial} of $\cB$ is 
$$\TIP(\cB, x):=
1+\sum_{i\in[\alpha(\cB)]}c_i(\cB)\cdot x_q^i\in \Z[x].$$
\end{definition}
Note that $\TIP(\cB, x)$ is an integer polynomial, as by setting $y=q$ in $x_y^i$, $x_q^i$ is an integer polynomial. The reason for using $x_q^i$ in the definition of $\TIP(\cB, x)$ will be clear in the following.

\subsubsection{Totally-isotropic space polynomials as a $q$-analogue of independence polynomials.} 
We now turn to totally-isotropic space polynomials for graphical matrix spaces. 

Let $G$ be a graph and $\cG_q$ be the graphical matrix space of $G$. Recall that $\alpha(\cG_q)=\alpha(G)$ for any prime power $q$ by \cite{BCG+21}. This ensures that $\TIP(\cG_q, x)$ is of the same degree as $\IP(G, x)$. 

Our main result considerably strengthens the above connection. It shows that there exists a polynomial $\IP(G, x, y)\in \Z[y][x]$  
that captures simultaneously the cases of 
$\TIP(\cG_q, x)$ for prime power $y=q$ and $\IP(G, x)$ for $y=1$. The following result follows from Theorem~\ref{thm:main_enum}, the definition of $\TIP(\cB, x)$, and $x_1^d=x^d$.
\begin{theorem}\label{thm:main}
	Let $x$ and $y$ be variables. For a graph $G=([n], E)$, there exists a polynomial  
	$$\IP(G, x, y)=1+\sum_{i\in[\alpha(G)]}c_i(y)\cdot x_y^i\in \Z[y][x],$$ 
	where $c_i(y)\in\Z[y]$, such that $\IP(G, x, y)$ satisfies the following: 
	\begin{itemize}
		\item when $y=q$ is a prime power, $\IP(G, x, q)=\TIP(\cG_q, x)$;
		\item when $y=1$, $\IP(G, x, 1)=\IP(G, x)$.
	\end{itemize}
\end{theorem}

We now explain the reason for using $x_y^d$ as a basis of $\Z[y][x]$.

Let $\cB\leq \Lambda(n_1, q)$ and $\cC\leq \Lambda(n_2, q)$. The \emph{disjoint direct sum} of $\cB$ and $\cC$ is $\cA=\{\begin{bmatrix}
	B & 0 \\
	0 & C
	\end{bmatrix}\mid B\in \cB, C\in \cC\}\leq \Lambda(n_1+n_2, q)$. 
The disjoint direct sum of two alternating matrix spaces can be viewed as corresponding to the union of two graphs. Indeed, by \cite{LQ20}, a graph $G$ is connected if and only if its associated graphical matrix space cannot be written as the disjoint direct sum of two alternating matrix spaces.


The following proposition corresponds to Equation~\ref{eq:graph_union} in the graph setting. Its proof is in Section~\ref{sec:direct-sum}. 
\begin{proposition}\label{prop:direct-sum}
Let $\cB\leq\Lambda(n_1, q)$, $\cC\leq\Lambda(n_2, q)$, and $\cA\leq \Lambda(n_1+n_2, q)$ be the disjoint direct sum of $\cB$ and $\cC$. Then 
$$\TIP(\cA, x)=\TIP(\cB, x)\cdot\TIP(\cC, x).$$
\end{proposition}

\subsection{Group-theoretic interpretations.}\label{subsec:group}

\subsubsection{A connection between alternating matrix spaces and groups.}\label{subsec:connection}
Let $p$ be a prime $>2$. Alternating matrix spaces over $\F_p$ are closely related to finite $p$-groups of class $2$ and exponent $p$ via Baer's correspondence \cite{Bae38}. That is, from such a group $P$ one can construct an alternating matrix space $\cB_P$. We briefly review this process here. 

Let $P$ be a $p$-group of class $2$ and exponent $p$. By taking the commutator map in $P$, we obtain an alternating bilinear map $\phi_P:P/[P, P]\times P/[P, P]\to[P, P]$. As $P/[P,P]$ and $[P, P]$ are elementary abelian groups, we can set $P/[P,P]\cong\Z_p^n$ and $[P, P]\cong \Z_p^m$. Let us fix a basis of $P/[P,P]$ as $\{e_1, \dots, e_n\}$, and a basis of $[P,P]$ as $\{f_1, \dots, f_m\}$. For each $f_i$, from $\phi_P$ we obtain an alternating matrix $F_i\in \Lambda(n, p)$, where the $(j, k)$th entry of $F_i$ is the coefficient of $f_i$ in $\phi_P(e_j, e_k)$. Therefore, by fixing bases of $P/[P,P]$ and $[P,P]$, one can obtain an $m$-tuple of $n\times n$ alternating matrices over $\F_p$, whose linear span is denoted by $\cB_P\leq \Lambda(n, p)$. 

Conversely, from an alternating matrix space $\cB$ one can construct such a group $P_\cB$. For a detailed description of this procedure we refer readers to \cite{LQ20,HQ21}.

%

\subsubsection{Abelian subgroups and totally-isotropic spaces.}

Let \(P\) be a \(p\)-group of class \(2\) and exponent \(p\). Since \(P\) has exponent \(p\), every abelian subgroup of \(P\) is elementary abelian, that is, isomorphic to \(\mathbb{Z}_p^k\) for some \(k \in \mathbb{N}\). Moreover, if \(B\) is an abelian subgroup of \(P\), then $
B=B'\times B''$, 
where \(B''=B\cap [P,P]\) and \(B'\cap [P,P]=1\). Since \([P,P]\le Z(P)\), it follows that \(B'[P,P]\) is again abelian. Hence every abelian subgroup of \(P\) is contained in an abelian subgroup containing the commutator subgroup.

It follows that totally isotropic subspaces of \(\mathcal B_P\) correspond to abelian subgroups of \(P\) that contain the commutator subgroup. This observation was used by Alperin to construct large abelian subgroups of \(p\)-groups of class \(2\) and exponent \(p\) \cite{Alp65}. It is also the starting point of work of Ol'shanskii \cite{Ols78} and of Buhler, Gupta and Harris \cite{BGH87} on the construction of \(p\)-groups with small maximal abelian subgroups.

\subsubsection{Enumeration polynomials of abelian subgroups containing the commutator subgroup.}
Our totally isotropic space polynomials can be viewed as enumeration polynomials for abelian subgroups of \(p\)-groups of class \(2\) and exponent \(p\) containing the commutator subgroup.

Let \(P\) be a \(p\)-group of class \(2\) and exponent \(p\). Let $A_i(P)$ be the set of abelian subgroups $S$ in $P$ containing $[P, P]$, with $S/[P,P]$ being of order $p^i$, and let $a_i(P):=|A_i(P)|$. Let $x$ be a variable, and $\alpha(P):=\max\{i\in\N\mid A_i(P)\neq\emptyset\}$. The enumeration polynomial of abelian subgroups of $P$ containing the commutator subgroup is naturally
$$
\AP(P, x):=1+\sum_{i\in[\alpha(P)]}a_i(P)\cdot x_p^i \in\Z[x].
$$

Let $\cB_P$ be the alternating matrix space associated with $P$ via the construction in Section~\ref{subsec:connection}. We then have $\AP(P, x)=\TIP(\cB_P, x)$. Note that while our construction of $\cB_P$ from $P$ is basis-dependent, the resulting polynomial $\TIP(\cB_P, x)$ does not depend on the basis choices.

Furthermore, disjoint direct sums of alternating matrix spaces correspond to direct products of groups \cite{Wil09b}. Let $P$ and $Q$ be two $p$-groups of class $2$ and exponent $p$, and $P\times Q$ be their direct product. Then Proposition~\ref{prop:direct-sum} gives us that
$
\AP(P, x)\cdot \AP(Q, x)=\AP(P\times Q, x).
$

\subsection{Rossmann's work, outlook, and open questions.}

\subsubsection{Rossmann's work \cite{Ros22} and graphical groups.}
A closely related work is by T. Rossmann, who studied enumeration functions of conjugacy classes of graphical groups over finite fields \cite{Ros22}. 
Given a graph $G$ and a commutative ring $R$, one can construct an associated graphical group. When $R=\mathbb{F}_p$ with $p>2$, this group is a $p$-group of class $2$ and exponent $p$.
In \cite{Ros22}, Rossmann first introduced a new family of graph polynomials in two variables. He then showed that specialising the first variable gives rise to the enumeration polynomial of conjugacy classes of graphical groups over finite fields \cite[Theorem A]{Ros22}. 

One consequence is that the number of size-$e$ conjugacy classes of graphical groups over finite fields of order $q$ is given by a polynomial in $q-1$ with integer coefficients \cite[Corollary C]{Ros22}. Analogously, our Theorem~\ref{thm:main_enum} can be interpreted as follows: for graphical groups over the prime field $\mathbb F_p$, with $p>2$, the number
of rank-$i$ elementary abelian subgroups containing the commutator subgroup is
given by evaluating an integer polynomial at $p$.

We note that the proof of \cite[Theorem A]{Ros22}, as that of our Theorem~\ref{thm:main}, deals with alternating bilinear maps, or equivalently in this case, alternating matrix spaces. Indeed, putting in the language of alternating matrix spaces, for $\cB\leq\Lambda(n, q)$, the object to enumerate for \cite[Theorem A]{Ros22} is 
\begin{equation}\label{eq:rossmann}
	\{v\in\F_q^n\mid \dim(\{Bv\mid B\in \cB\})=e\}.
\end{equation}

Rossmann also showed that the product of conjugacy class enumeration polynomials of two graphical groups gives the conjugacy class enumeration polynomial of the graphical group corresponding to the disjoint union of the two graphs \cite[Proposition 5.1]{Ros22}. Rossmann did not need our $x_q^i$, because the objects to enumerate there are vectors (Equation~\ref{eq:rossmann}), instead of subspaces as in our case.

\subsubsection{Outlook and open questions.} Rossmann’s work \cite{Ros22}, together with the results of the present paper, suggest that enumerating certain structures in finite groups leads to enumeration polynomials generalising corresponding graph-theoretic enumeration polynomials. It is reasonable to expect that further correspondences of this kind remain to be explored. For example, \cite{BCG+21} established a correspondence between vertex colourings of graphs and direct sum decompositions into totally isotropic spaces of alternating bilinear maps. This raises the question of whether one may define a linear-algebraic analogue of the chromatic polynomials of graphs.



Computing graph independence polynomials is a natural problem, and it has been studied from several perspectives, including recursive identities and decomposition formulas \cite[Section~2]{GH83}, as well as algorithmic approaches via the hard-core model partition function \cite{Wei06}. While the proof of Theorem~\ref{thm:main} yields a method for computing TI polynomials of graphical matrix spaces through analysing the underlying graph structure, it would be of interest to develop more efficient methods, either through recursive identities or from an algorithmic viewpoint. 

Overall, it could be of interest to further the study of $I(G, x, q)$ by transferring questions and methodologies for the independence polynomials such as in \cite{GH83}. We leave these to future works.

\section{Proof of Theorem~\ref{thm:main_enum}}\label{sec:main}

Our goal is to show that for any $1\leq i\leq \alpha(G)$, there exists $c_i(q)\in\Z[q]$, such that $c_i(q)=c_i(\cG_q)$ for any prime power $q$, and $c_i(1)=c_i(G)$. Briefly speaking, this requires us to identify some combinatorial properties of $G$ which essentially determine  $c_i(\cG_q)$. 

In the following we use $e_i$ to denote the $i$th standard basis vector of $\F_q^n$.

\paragraph{Some subsets of Grassmannians.} Let $\Gr(i, n, q)$ be the Grassmannian of $i$-dimensional subspaces of $\F_q^n$. 
We will consider the following subsets of $\Gr(i,n,q)$, each arising as the set of $\F_q$-rational points of a subvariety of the Grassmannian.


Let $\binom{[n]}{i}$ be the set of size-$i$ subsets of $[n]$. The natural total
order of $[n]$ induces the lexicographic total order on $\binom{[n]}{i}$.

For any $i$-dimensional subspace $U\le \F_q^n$, there is a unique matrix
$T_U\in \F_q^{n\times i}$ in reduced column echelon form whose columns form a basis of $U$.
Let $P_U\in \binom{[n]}{i}$ be the set of pivot-row indices of $T_U$.
Equivalently, if the Pl\"ucker coordinates of $U$ are indexed by elements of $\binom{[n]}{i}$,
then $P_U$ is the lexicographically first index at which the Pl\"ucker coordinate is non-zero. In general, if $T$ is any $n \times i$ matrix whose columns span $U$, then $P_U$ indexes the lexicographically first full-rank minor of $T$.

For $P\in \binom{[n]}{i}$, define
$$
\Gr(i,n,q)_P:=\{U\le \F_q^n\mid P_U=P\}.
$$
Thus $\Gr(i,n,q)_P$ consists of those $i$-dimensional subspaces whose reduced column echelon form has pivot rows indexed by $P$.

Fix $P\in\binom{[n]}{i}$. For any $U\in \Gr(i,n,q)_P$, let $T_U\in \F_q^{n\times i}$
be the unique matrix in reduced column echelon form whose columns form a basis of $U$.
Then $P_U=P$ is the set of pivot-row indices of $T_U$.
Define $Q_U\subseteq [n]\setminus P_U$ to be the set of non-pivot row indices at which $T_U$
has a non-zero row, namely
$$
Q_U:=\{j\in [n]\setminus P_U \mid \text{the $j$th row of }T_U\text{ is non-zero}\}.
$$
Equivalently,
$$
Q_U=\bigl\{j\in [n]\setminus P_U \mid U\not\subseteq \span\{e_k\mid k\in [n],\, k\neq j\}\bigr\}.
$$

For $P\in \binom{[n]}{i}$ and $Q\subseteq [n]\setminus P$, define
$$
\Gr(i,n,q)_{P,Q}:=\{U\in \Gr(i,n,q)\mid P_U=P,\ Q_U=Q\}.
$$

The following example illustrates that $P_U$ records the pivot rows of $T_U$, while $Q_U$ records the non-pivot rows of $T_U$ that are non-zero.

\begin{example}
	Suppose $q\geq 5$. Let $n=5$ and $i=2$, and let $U\le \F_q^5$ be the column space of
	$$
	T_U=
	\begin{pmatrix}
		1&0\\
		2&4\\
		0&1\\
		0&0\\
		3&5
	\end{pmatrix}.
	$$
	This matrix is in reduced column echelon form, and its pivot rows are $1$ and $3$. Hence
	$$
	P_U=\{1,3\}.
	$$
	Among the non-pivot rows $2,4,5$, the $2$nd and $5$th rows are non-zero, while the $4$th row is zero. Therefore
	$$
	Q_U=\{2,5\}.
	$$
	Thus $U\in \Gr(2,5,q)_{P,Q}$ for $P=\{1,3\}$ and $Q=\{2,5\}$.
\end{example}

%

\paragraph{Relating independent sets with $\Gr(i, n, q)_P$.} Let $G=([n], E)$ be a graph. For $1\leq i\leq \alpha(G)$, let $S_i(G)$ be the set of size-$i$ independent sets of $G$. For a prime power $q$, let $S_i(\cG_q)$ be the set of $i$-dimensional totally-isotropic spaces of $\cG_q$. Let $\Gr(i, n, q, G):=\Gr(i, n, q)\cap S_i(\cG_q)$. 

For $P\in\binom{[n]}{i}$, let $\Gr(i, n, q, G)_P=\Gr(i, n, q)_P\cap S_i(\cG_q)$. 
\begin{claim}\label{claim:P}
$\Gr(i, n, q, G)_P\neq\emptyset$ if and only if $P$ is an independent set.
\end{claim}
\begin{proof}
If $P$ is a size-$i$ independent set $P$ of $G$, $\Gr(i, n, q)_P\cap S_i(\cG_q)$ is non-empty, as it contains the subspace $\span\{e_i\mid i\in P\}\leq\F_q^n$. 

Suppose $U\in\Gr(i, n, q)$ is a totally-isotropic space of $\cG_q$. As already observed in \cite[Section 5, proof of Theorem 1.3]{BCG+21}, it can be verified easily that, if the Pl\"ucker coordinate of $U$ at $R\in\binom{[n]}{i}$ is non-zero, then $R$ is a size-$i$ independent set of $G$. It follows that if $P\in\binom{[n]}{i}$ is not an independent set, then $\Gr(i, n, q)_P\cap S_i(\cG_q)=\emptyset$.
\end{proof}
By Claim~\ref{claim:P}, we have $\Gr(i, n, q, G)$ is a disjoint union of $\Gr(i, n, q, G)_P$ over $P\in S_i(G)$.

\paragraph{Relating graph structures with $\Gr(i, n, q)_{P, Q}$.} Let $P\in S_i(G)$, so $\Gr(i, n, q, G)_P$ is non-empty. 

Let $\Gr(i, n, q, G)_{P, Q}:=\Gr(i, n, q)_{P, Q}\cap S_i(\cG_q)$. We wish to understand when $\Gr(i, n, q, G)_{P, Q}$ is empty. This is easy for $Q=\emptyset$: in this case, $\Gr(i, n, q, G)_{P, \emptyset}=\{\span\{e_i\mid i\in P\}\}$. 

In general, for $P, Q\subseteq[n]$, let $G[P\cup Q]$ be the induced subgraph of $G$ on $P\cup Q$. Fix $Q\subseteq[n]\setminus P$. Take some $U\in \Gr(i, n, q)_{P, Q}$. Let $T$ be an $n\times i$ matrix whose columns form a basis of $U$. For $u\in [n]$, let $r_u$ be the $u$th row of $T$. 
\begin{lemma}\label{obs:easy}
Let $U\in\Gr(i, n, q)_{P, Q}$ and $T$ be as above. Then $U$ is a totally-isotropic space for $\cG_q$ if and only if for any edge	$\{u, v\}$ in $G[P\cup Q]$, $r_u$ and $r_v$ are linearly dependent. 
\end{lemma}
\begin{proof}
Recall that $G=([n], E)$. To start with, note that $U$ is a totally-isotropic space of $\cG_q$ if and only if for any $\{u, v\}\in E$, 
\begin{equation}\label{eq:all-zero}
	\tr{r_u}r_v-\tr{r_v}r_u=0_{i\times i}, 
\end{equation} 
where $0_{i\times i}$ denotes the $i\times i$ all-zero matrix. (Recall that $r_u$ and $r_v$ are row vectors.) 

If $u\not\in P\cup Q$, then $r_u$ is the zero vector, so Equation~\ref{eq:all-zero} is satisfied. The same holds for $v$. Therefore, we only need to consider $u, v\in P\cup Q$. Then note that Equation~\ref{eq:all-zero} just expresses that $r_u$ and $r_v$ are linearly dependent, concluding the proof.
\end{proof}
\begin{claim}\label{claim:Q}
	Let $u, v\in P$, $u\neq v$. If $G[P\cup Q]$ contains a path connecting $u$ and $v$, then $\Gr(i, n, q, G)_{P, Q}=\emptyset$.
\end{claim}
\begin{proof}
For the sake of contradiction, suppose $\Gr(i, n, q, G)_{P, Q}$ is non-empty. Take an $n\times i$ matrix $T$ whose column span is in $\Gr(i, n, q, G)_{P, Q}$. By Lemma~\ref{obs:easy}, rows of $T$ corresponding to the vertices on this path are all linearly dependent. It follows that row $u$ and row $v$ are linearly dependent, which is not possible because $u, v\in P$. 
\end{proof}

Claim~\ref{claim:Q} implies that for $\Gr(i, n, q, G)_{P, Q}$ to be non-empty, every connected component of $G[P\cup Q]$ contains at most one $u\in P$. 

\begin{claim}\label{claim:Q1}
	Let $C\subseteq P\cup Q$ be a connected component of $G[P\cup Q]$, and suppose $C$ contains exactly one $u\in P$. If there exists $v\in C$ with $v< u$, then $\Gr(i, n, q, G)_{P, Q}=\emptyset$.
\end{claim}
\begin{proof}
	For the sake of contradiction, suppose $\Gr(i, n, q, G)_{P, Q}$ is non-empty. Take an $n\times i$ matrix $T$ whose column span $U$ is in $\Gr(i, n, q, G)_{P, Q}$. For $w\in [n]$, let $r_w$ be the $w$th row of $T$. 
	Note that $v\in Q$, as $u$ is the only vertex of $C$ lying in $P$. Furthermore, $r_v$ is non-zero, and it is linearly dependent with $r_u$. So if $v<u$, we can replace $u$ with $v$ in $P$, so $P$ cannot be the lexicographic-first non-zero Pl\"ucker coordinate of $U$.
\end{proof}

\begin{claim}\label{claim:Q0}
Let $D\subseteq P\cup Q$ be a connected component of $G[P\cup Q]$, and suppose $D$ does not contain $u\in P$. If $\min\{v\mid v\in D\}<\min\{u\mid u\in P\}$, then $\Gr(i, n, q, G)_{P, Q}=\emptyset$.
\end{claim}
\begin{proof}
	For the sake of contradiction, suppose $\Gr(i, n, q, G)_{P, Q}$ is non-empty. Take an $n\times i$ matrix $T$ whose column span $U$ is in $\Gr(i, n, q, G)_{P, Q}$. As $D$ does not contain $u\in P$, we have $D\subseteq Q$.
	If $\min\{v\mid v\in D\}<\min\{u\mid u\in P\}$, then the lexicographically first non-zero row of $T$ is indexed by some $v'\in Q\setminus P$. It follows $P$ cannot be the lexicographic-first non-zero Pl\"ucker coordinate of $U$.
\end{proof}

The conditions in Claims~\ref{claim:Q}, \ref{claim:Q1}, \ref{claim:Q0} can then be used to deduce a characterisation of non-empty $\Gr(i, n, q, G)_{P, Q}$.
\begin{lemma}\label{lem:Q}
$\Gr(i, n, q, G)_{P, Q}\neq\emptyset$ if and only if $G[P\cup Q]$ satisfies the following: 
\begin{enumerate}
	\item every connected component of $G[P\cup Q]$ contains at most one $u\in P$, 
	\item any connected component $C$ with one $u\in P$ satisfies $v\geq u$ for $v\in C$, 
	\item and any connected component $D$ with no $u\in P$ satisfies $\min\{v\mid v\in D\}> \min\{u\mid u\in P\}$.
\end{enumerate}

Furthermore, when $\Gr(i,n,q,G)_{P,Q}\neq\emptyset$ and $Q\neq\emptyset$, 
$\lvert \Gr(i,n,q,G)_{P,Q}\rvert$ can be expressed as a non-empty and finite product of factors of the form $q^e-1$, with $e\in\N$.
\end{lemma}
\begin{proof}
The only if direction has been shown by Claims~\ref{claim:Q}, \ref{claim:Q1}, and~\ref{claim:Q0}. For the if direction, we can construct $U\in \Gr(i, n, q, G)_{P, Q}$ when $G[P\cup Q]$ satisfies the conditions in the statement. This construction process also gives the number of such $U$.

Recall that 
$$\Gr(i, n, q, G)_{P, Q}=\Gr(i, n, q)_{P, Q}\cap S_i(\cG_q).$$ Subspaces in $\Gr(i, n, q)_{P, Q}$ are in one-to-one correspondence with matrices $T$ of size $n\times i$, such that the submatrix of $T$ indexed by $P$ is the identity matrix, and the non-zero rows of $T$ are in $P\cup Q$. For $u\in [n]$, let $r_u$ be the $u$th row of $T$. We turn to examine those $T$ whose column span is totally-isotropic for $\cG_q$, and count the numbers of such $T$. 



By Lemma~\ref{obs:easy}, the column span of $T$ is totally-isotropic for $\cG_q$ if and only if for any edge $\{u, v\}\in G[P\cup Q]$, $r_u$ and $r_v$ are linearly dependent. This is the condition we will need to keep track of during the construction. Now consider the following two cases.
\begin{enumerate}
	\item[(a)] For a connected component $C$ with $u\in P$, note that $C$ contains exactly one $u\in P$, as $T$ is assumed to be the identity on the rows indexed by $P$. Consider $v\in C$, $v\neq u$. Then we must have $r_v=\alpha\cdot r_u$ with non-zero $\alpha\in\F_q$. This gives $(q-1)^{|C|-1}$ possibilities for $r_v$, $v\neq u$, $v\in C$. 
	\item[(b)] For a connected component $D$ with no $u\in P$, let $v^*=\min\{v\mid v\in D\}$. Suppose $|\{u\in P\mid u<v^*\}|=d$. Note that $d\geq 1$. In this case, there are $q^d-1$ possibilities for a non-zero $r_{v^*}$, because $r_{v^*}$ must be spanned by $\{r_u : u \in P,\, u < v^*\}$. For other $v\in D$, $r_v=\alpha\cdot r_{v^*}$, so there are $(q-1)^{|D|-1}$ possibilities. 
\end{enumerate}

It can be seen that the resulting matrix $T$ spans a totally-isotropic space, as for any edge $\{u, v\}\in G[P\cup Q]$, $r_u$ and $r_v$ are linearly dependent. Furthermore, the above assignments of $r_v$ for $v\in C$, $v\neq u$ and $r_v$ for $v\in D$ give rise to, and cover, $T\in \Gr(i, n, q)_{P, Q}$ corresponding to totally-isotropic $U\in\Gr(i, n, q, G)_{P, Q}$. That is, there are two types of restrictions: the restriction from the graph $G[P\cup Q]$, and the restriction of the lexicographic order. These two types of restrictions together give a product of factors of the form
$q^e-1$.


Finally, note that if $Q\neq\emptyset$, then at least one of the following two cases must hold, namely case (a) with $|C|>1$, or case (b) with $|D|\geq 1$. In the former case, there exists $v\in C$, $v\neq u$, such that $r_v=\alpha\cdot r_u$, leading to a factor of $(q-1)$. In the latter case, there is at least a factor of the form $q^d-1$. That is, some non-trivial $q^e-1$ factor appears as long as $Q\neq \emptyset$. The proof is concluded. 
\end{proof}

\paragraph{Concluding the proof of Theorem~\ref{thm:main_enum}.} We have seen that $\Gr(i, n, q, G)$ is a disjoint union of $\cup_{P\in S_i(G)}\Gr(i, n, q, G)_P$. 

For $P\in S_i(G)$, let 
$$\mathcal Q_P=\{Q\subseteq[n]\setminus P\mid G[P\cup Q] \text{ satisfies (1), (2) and (3) in Lemma~\ref{lem:Q}}\}.$$
Then we have
$$\Gr(i, n, q, G)_P=\bigcup_{Q\in \mathcal Q_P} \Gr(i, n, q, G)_{P, Q}.$$
We now distinguish between two cases. First, 
when $Q=\emptyset$, it is easy to observe that $|\Gr(i, n, q, G)_{P, \emptyset}|=1$. Second, when $Q\neq\emptyset$, by Lemma~\ref{lem:Q}, $|\Gr(i, n, q, G)_{P, Q}|$ is a non-empty product of $q^e-1$ where the exponents only depend on $P$, $Q$, and the graph structure $G[P\cup Q]$. Therefore, for $P\in S_i(G)$, 
$$|\Gr(i, n, q, G)_{P}|=1+\sum_{Q\in\mathcal Q_P, Q\neq \emptyset} \prod_{e_i}(q^{e_i}-1)$$ 
for any $q$. In particular, $|\Gr(i, n, 1, G)_{P}|=1$. This concludes the proof of Theorem~\ref{thm:main_enum}. \qed

\section{Proof of Proposition~\ref{prop:direct-sum}}\label{sec:direct-sum}

The goal of this section is to prove Proposition~\ref{prop:direct-sum}. We first prepare for the proof.

We use $\qbinom{n}{k}{q}$ to denote the Gaussian binomial coefficient counting the number of $k$-dimensional subspaces of $\F_q^n$. That is, 
$$
\qbinom{n}{k}{q}=\frac{(q^n-1)\cdot(q^n-q)\cdot\ldots\cdot(q^n-q^{k-1})}{(q^k-1)\cdot(q^k-q)\cdot\ldots\cdot(q^k-q^{k-1})}.
$$

We also need to work with the following setting. Let $X_1$ be the subspace of $\F_q^{n_1+n_2}$ spanned by the first $n_1$ standard basis vectors, and let $X_2$ be the subspace of $\F_q^{n_1+n_2}$ spanned by the last $n_2$ standard basis vectors. Let $\pi_1$ be the projection from $\F_q^{n_1+n_2}$ onto $X_1$ along $X_2$, and let $\pi_2$ be the projection from $\F_q^{n_1+n_2}$ onto $X_2$ along $X_1$. 

Let $U\leq X_1$ and $V\leq X_2$, with $\dim(U)=d$ and $\dim(V)=e$. We are interested in subspaces $W\leq \F_q^{n_1+n_2}$ such that $\pi_1(W)=U$ and $\pi_2(W)=V$. Note that $\max\{d, e\}\leq \dim(W)\leq d+e$. Let $s:=d+e-\dim(W)$, so $0\leq s\leq \min\{d, e\}$. 

Given $d, e, 
s\in \N, s\leq \min\{d, e\}$, we define 
\begin{equation}\label{eq:Cdesq}
C_{d, e, s, q}=\qbinom{d}{s}{q}\cdot 
\qbinom{e}{s}{q}\cdot (q^s-1)\cdot \ldots\cdot (q^s-q^{s-1}).
\end{equation}
When $q$ is obvious from the context, we may simply write $C_{d, e, s}$ 
instead of $C_{d, e, s, q}$. 

We then show the following counting result; see Section~\ref{subsec:des} for its proof.

\begin{lemma}\label{lem:Cdesq}
	The number of $(d+e-s)$-dimensional $W$ with $\pi_1(W)=U$ and $\pi_2(W)=V$ is $C_{d, e, s, q}$.
\end{lemma}

We now relate $C_{d,e,s,q}$ with $x_q^d$ introduced in Equation~\ref{eq:xyd}; see Section~\ref{subsec:Ed} for its proof.
\begin{lemma}\label{lem:Ed}
For $d, e\in\N$,  we have 
\begin{equation}\label{eq:xqd}
x_q^d\cdot x_q^e=\sum_{s=0}^{\min\{d, e\}}C_{d, e, s, q}\cdot x_q^{d+e-s}.
\end{equation}
\end{lemma}

We now prove Proposition~1.5 using Lemmas~3.1 and~3.2.

\paragraph{Proposition~\ref{prop:direct-sum}, restated.}
Let $\cB\leq\Lambda(n_1, q)$, $\cC\leq\Lambda(n_2, q)$, and $\cA\leq \Lambda(n_1+n_2, q)$ be the disjoint direct sum of $\cB$ and $\cC$. Then $\TIP(\cA, x)=\TIP(\cB, x)\cdot\TIP(\cC, x)$.
\vskip 1em

\begin{proof}[Proof of Proposition~\ref{prop:direct-sum}]
	
	
	Let $X_1$ be the subspace of $\F_q^{n_1+n_2}$ spanned by the first $n_1$ standard basis vectors, and let $X_2$ be the subspace of $\F_q^{n_1+n_2}$ spanned by the last $n_2$ standard basis vectors. Let $\pi_1$ be the projection from $\F_q^{n_1+n_2}$ onto $X_1$ along $X_2$, and let $\pi_2$ be the projection from $\F_q^{n_1+n_2}$ onto $X_2$ along $X_1$. 
	
	Let $U\leq X_1$ be a totally-isotropic space of $\cB$, and $V\leq X_2$ be a totally-isotropic space of $\cC$. Suppose $\dim(U)=d$ and $\dim(V)=e$. Therefore, $U$ contributes $x_q^d$ in $\TIP(\cB, x)$, and $V$ contributes $x_q^e$ in $\TIP(\cC, x)$. 
	
	As $\cA$ is the disjoint direct sum of $\cB$ and $\cC$, it can be verified that $W$ is totally-isotropic for $\cA$ if and only if $\pi_1(W)$ is totally-isotropic for $\cB$ and $\pi_2(W)$ is totally-isotropic for $\cC$. 
	
	Lemma~\ref{lem:Cdesq} shows that $C_{d, e, s, q}$ is the number of totally-isotropic spaces $W$ such that $\pi_1(W)=U$, $\pi_2(W)=V$, $\dim(W)=d+e-s$. Lemma~\ref{lem:Ed} shows that the number $C_{d, e, s, q}$ is the coefficient of $x_q^{d+e-s}$ in the expansion of $x_q^d\cdot x_q^e$. The proof is then concluded. 
	%
\end{proof}
After the proof, it may be instructive to examine an example. 
\begin{example}
Let $\cB\leq\Lambda(n_1, q)$ and $\cC\leq\Lambda(n_2, q)$. Let $\cA\leq\Lambda(n_1+n_2, q)$ be the disjoint direct sum of $\cB$ and $\cC$. Let $U_1=\span\{u_1\}$ be a $1$-dimensional subspace of $\F_q^{n_1}$ and $U_2=\span\{u_2\}$ a $1$-dimensional subspace of $\F_q^{n_2}$. By the alternating property, $U_1$ and $U_2$ are totally-isotropic spaces of $\cB$ and $\cC$, respectively. 

Suppose we are interested in $W\leq\F_q^{n_1+n_2}$ such that $\pi_1(W)=U_1$ and $\pi_2(W)=U_2$, where $\pi_1$ is the projection to the first $n_1$ coordinates, and $\pi_2$ is the projection to the last $n_2$ coordinates. Note that such $W$ is a totally-isotropic space for $\cA$. We then see that $\dim(W)=2$ or $\dim(W)=1$. When $\dim(W)=2$, such $W$ is unique. When $\dim(W)=1$, then $W=\span\{w\}\in \F_q^{n_1+n_2}$ for some non-zero $w\in \F_q^{n_1+n_2}$. Suppose $w=\begin{bmatrix}
	w_1\\
	w_2
\end{bmatrix}$, where $w_1\in\F_q^{n_1}$. We can then fix $w_1=u_1$, and after that, $w_2$ can be any non-zero scalar multiple of $u_2$. It follows that there are $(q-1)$ such $1$-dimensional $W$. 

We then see that this is consistent with the choices of $x_q^i$. That is, $x_q\cdot x_q=x^2=x(x-(q-1))+(q-1)\cdot x=x_q^2+(q-1)\cdot x_q$.
\end{example}

\subsection{Proof of Lemma~\ref{lem:Cdesq}}\label{subsec:des}
	We need to count the number of $W\leq\F_q^{n_1+n_2}$ such that
	\begin{equation}\label{eq:W}
	 \dim(W)=d+e-s, \pi_1(W)=U, \text{ and } \pi_2(W)=V.
	 \end{equation}
	Let $K_1=\ker(\pi_2)\cap W=X_1\cap W$, and $K_2=\ker(\pi_1)\cap W=X_2\cap W$.
	By $\dim(W)=\dim(V)+\dim(K_1)$, we have $\dim(K_1)=\dim(W)-\dim(V)=d+e-s-e=d-s$. Similarly, we have $\dim(K_2)=e-s$.
	
	Take any $W\leq \F_q^{n_1+n_2}$ satisfying Equation~\ref{eq:W}. Let $T\in\M((n_1+n_2)\times (d+e-s), q)$ be a matrix whose columns span $W$. By arranging an appropriate basis of $W$, we can set 
	\begin{equation}\label{eq:T}
		T=\begin{bmatrix}
			T_1 & T_2 & 0_{n_1\times (e-s)} \\
			0_{n_2\times (d-s)} & T_3 & T_4
		\end{bmatrix},
	\end{equation}
	where $T_2\in \M(n_1\times s, q)$ and $T_3\in \M(n_2\times s, q)$. We then see that the columns of $\begin{bmatrix}
		T_1 \\
		0_{n_2\times (d-s)}
	\end{bmatrix}$ span $K_1$, the columns of 
	$\begin{bmatrix}
		0_{n_1\times (e-s)} \\
		T_4
	\end{bmatrix}$ span $K_2$, the columns of 
	$\begin{bmatrix}
		T_2 \\
		0_{n_2\times s}
	\end{bmatrix}$ span $L_1$ which is a complementary subspace of $K_1$ in $U$, and 
	the columns of 
	$\begin{bmatrix}
		0_{n_1\times s} \\
		T_3
	\end{bmatrix}$ span $L_2$ which is a complementary subspace of $K_2$ in $V$. That is, $U=K_1\oplus L_1$ and $V=K_2\oplus L_2$.
	

	Any subspace $W$ satisfying Equation~\ref{eq:W} has an ordered linear basis as columns of a matrix $T$ in Equation~\ref{eq:T}. On the other hand, every such $T$ gives rise to such a subspace $W$. However, two different $T$ and $T'$ may give rise to the same $W$. It is clear that for $T$ and $T'$ to give rise to the same $W$, then $T$ and $T'$ must be related by the following transformations: 
	\begin{enumerate}
		\item Elementary operations on the first $d-s$ columns. This means that to avoid double counting, we need to count the number of subspaces spanned by $\begin{bmatrix}
			T_1 \\
			0_{n_2\times (d-s)}
		\end{bmatrix}$. 
		\item Elementary operations on the last $e-s$ columns. This means that to avoid double counting, we need to count the number of subspaces spanned by $\begin{bmatrix}
			0_{n_1\times (e-s)} \\
			 T_4
		\end{bmatrix}$. 
		\item Adding columns from the first $d-s$ columns and the last $e-s$ columns to the middle $s$ columns. This means that to avoid double counting, we don't need to enumerate complementary subspaces of $K_1$ in $U$, but any representative complementary subspace would suffice. Similarly, we don't need to enumerate complementary subspaces of $K_2$ in $V$.
		\item Elementary operations on the middle $s$ columns. This means that to avoid double counting, we can fix an ordered basis of $L_1$, that is the $T_2$ part in Equation~\ref{eq:T}. Note that this leaves $T_3$ part undetermined. 
	\end{enumerate} 
	
	
	Summarising the above, we proceed with the counting. To start with, by  (1) and (2), the choices of $K_1$ and $K_2$ are uniquely determined by $W$, and there are $\qbinom{d}{s}{q}\cdot \qbinom{e}{s}{q}$ such choices. Once these $K_1$ and $K_2$ are fixed, by (3), we can fix $L_1$ and $L_2$ as complementary subspaces of $K_1$ in $U$ and $K_2$ in $V$, respectively. 
	By (4), the $T_3$ part needs to be an ordered basis of $L_2$, and the number of ordered bases is $(q^s-1)\cdot \ldots\cdot (q^s-q^{s-1})$.  
	
	Putting these together, we get the number of such subspaces $W$ as 
	$$\qbinom{d}{s}{q}\cdot \qbinom{e}{s}{q}\cdot (q^s-1)\cdot \ldots\cdot (q^s-q^{s-1})= C_{d, e, s, q}.$$ 
	This concludes the proof of Lemma~\ref{lem:Cdesq}. \qed

\subsection{Proof of Lemma~\ref{lem:Ed}}\label{subsec:Ed} By switching the role of $d$ and $e$ in the following if necessary, we can assume $e\leq d$. Proof by induction on $e$. If $e=0$, then this holds trivially. 

Consider the case \(e\geq 1\), and assume that the claim holds for all integers from \(1\) up to \(e-1\).
By 
induction, we have 
\begin{eqnarray*}
x_q^d\cdot x_q^e & = & x_q^d\cdot x_q^{e-1}\cdot (x-(q^{e-1}-1)) \\
& = & \left(\sum_{s=0}^{e-1}C_{d, e-1, s}\cdot x_q^{d+e-1-s}\right)\cdot \left(x-(q^{e-1}-1)\right).
\end{eqnarray*}
For $i\in\{0, 1, \dots, e-1\}$, we have 
\begin{eqnarray}
& & C_{d, e-1, i}\cdot x_q^{d+e-1-i}\cdot (x-(q^{e-1}-1)) \nonumber\\
& = & C_{d, e-1, i}\cdot x_q^{d+e-1-i}\cdot 
(x-(q^{d+e-1-i}-1) \nonumber \\
& & \quad +(q^{d+e-1-i}-1)-(q^{e-1}-1)) \nonumber\\
& = & C_{d, e-1, i}\cdot x_q^{d+e-1-i}\cdot (x-(q^{d+e-i-1}-1)) \nonumber\\
& & \quad +C_{d, e-1, i}\cdot 
x_q^{d+e-1-i}\cdot (q^{d+e-i-1}-q^{e-1}) \nonumber\\
& = & C_{d, e-1, i}\cdot x_q^{d+e-i}+C_{d, e-1, i}\cdot x_q^{d+e-1-i}\cdot 
q^{e-1}\cdot (q^{d-i}-1). \label{eq:ed}
\end{eqnarray}

First, set $i=0$ in the formula in Equation~\ref{eq:ed} to obtain
$$
C_{d, e-1, 0}\cdot x_q^{d+e}+C_{d, e-1, 0}\cdot x_q^{d+e-1}\cdot 
q^{e-1}\cdot (q^{d}-1).
$$
We then see that the coefficient of $x_q^{d+e}$ is $C_{d, e-1, 0}$. As $C_{d, e-1, 0}=C_{d, e, 0}$, we have $C_{d, e-1, 0}\cdot x_q^{d+e}=C_{d, e, 0}\cdot x_q^{d+e}$, giving us the term corresponding to $s=0$ in Equation~\ref{eq:xqd}.

Second, set $i=e-1$ in the formula in Equation~\ref{eq:ed} to obtain
$$
C_{d, e-1, e-1}\cdot x_q^{d+1}+C_{d, e-1, e-1}\cdot x_q^{d}\cdot 
q^{e-1}\cdot (q^{d-e+1}-1).
$$
We then see that the coefficient of $x_q^d$ is $C_{d, e-1, e-1}\cdot q^{e-1}\cdot (q^{d-e+1}-1)$, which is equal to $C_{d, e, e}$. This can be seen by setting $s=e$ in Equation~\ref{eq:Cdesq}, which gives $C_{d, e, e}=\qbinom{d}{e}{q}\cdot \qbinom{e}{e}{q}\cdot (q^e-1)\cdot\ldots\cdot (q^e-q^{e-1})=(q^d-1)\cdot\ldots\cdot(q^d-q^{e-1})=C_{d, e-1, e-1}\cdot(q^d-q^{e-1})=C_{d, e-1,e-1}\cdot q^{e-1}\cdot (q^{d-e+1}-1)$. 

Third, it remains to examine $x_q^{d+e-j}$ for $j\in \{1, \dots, e-1\}$. Its coefficient has two summands: one is $C_{d, e-1, j}$, coming from setting $i=j$ in Equation~\ref{eq:ed}, and the other is $C_{d, e-1, j-1}\cdot q^{e-1}\cdot (q^{d-j+1}-1)$, coming from setting $i=j-1$ in Equation~\ref{eq:ed}. Then it suffices to show the following. 


\begin{claim}\label{claim:Cdei}
	For $i\in \{0, 1, \dots, e-2\}$, we have
$$
C_{d, e-1, i}\cdot q^{e-1}\cdot (q^{d-i}-1)+C_{d, e-1, i+1}=C_{d, e, i+1}.
$$
\end{claim}
\begin{proof}
To start with, we have that 
\begin{eqnarray*}
& & C_{d, e-1, i}\cdot q^{e-1}\cdot (q^{d-i}-1)+C_{d, e-1, i+1} \\
& = & \qbinom{d}{i}{q}\cdot \qbinom{e-1}{i}{q}\cdot (q^i-1)\cdot\ldots\cdot 
(q^i-q^{i-1})\cdot q^{e-1}\cdot 
(q^{d-i}-1)\\
& & \quad +\qbinom{d}{i+1}{q}\cdot\qbinom{e-1}{i+1}{q}\cdot (q^{i+1}-1)\cdot 
\ldots\cdot (q^{i+1}-q^i)\\
& = & \qbinom{d}{i}{q}\cdot \qbinom{e-1}{i}{q}\cdot (q^{i+1}-q)\cdot\ldots\cdot 
(q^{i+1}-q^i)\cdot q^{e-1-i}\cdot 
(q^{d-i}-1)\\
& & \quad +\qbinom{d}{i+1}{q}\cdot\qbinom{e-1}{i+1}{q}\cdot (q^{i+1}-1)\cdot 
\ldots\cdot (q^{i+1}-q^i) \\
& = & (q^{i+1}-q)\cdot \ldots \cdot (q^{i+1}-q^i) \\
& & \quad
\cdot\Huge(
\qbinom{d}{i}{q}\cdot \qbinom{e-1}{i}{q}\cdot q^{e-1-i}\cdot (q^{d-i}-1)
+
\qbinom{d}{i+1}{q}\cdot\qbinom{e-1}{i+1}{q}\cdot (q^{i+1}-1)
\Huge).
\end{eqnarray*}
Finally, we note (1) $\qbinom{d}{i+1}{q}\cdot(q^{i+1}-1)=\qbinom{d}{i}{q} 
\cdot(q^{d-i}-1)$ and (2) a $q$-Pascal identity $\qbinom{e-1}{i}{q}\cdot 
q^{e-1-i}+\qbinom{e-1}{i+1}{q}=\qbinom{e}{i+1}{q}$. 
These allow us to obtain 
\begin{eqnarray*}
	& & \qbinom{d}{i}{q}\cdot \qbinom{e-1}{i}{q}\cdot q^{e-1-i}\cdot (q^{d-i}-1)
	+
	\qbinom{d}{i+1}{q}\cdot\qbinom{e-1}{i+1}{q}\cdot (q^{i+1}-1) \\
	& = & \qbinom{d}{i+1}{q}\cdot \qbinom{e-1}{i}{q}\cdot q^{e-1-i}\cdot (q^{i+1}-1)
	+
	\qbinom{d}{i+1}{q}\cdot\qbinom{e-1}{i+1}{q}\cdot (q^{i+1}-1) \\
	& = & \qbinom{d}{i+1}{q}\cdot (q^{i+1}-1)\cdot (\qbinom{e-1}{i}{q}\cdot 
	q^{e-1-i}+\qbinom{e-1}{i+1}{q})\\
	& = & \qbinom{d}{i+1}{q}\cdot (q^{i+1}-1)\cdot \qbinom{e}{i+1}{q}.
\end{eqnarray*}
This concludes the proof of Claim~\ref{claim:Cdei}.
\end{proof}
Letting $j=i+1$, we obtain $C_{d, e, j}=C_{d, e-1, j-1}\cdot q^{e-1}\cdot (q^{d-j+1}-1)+C_{d, e-1, j}$. This concludes the proof of Lemma~\ref{lem:Ed}.\qed

\subsubsection*{Acknowledgement.} I would like to thank Yuval Wigderson for his feedback on an early draft of this paper. I would like to express my sincere thanks to the anonymous reviewers for their careful and thoughtful reviews.

\bibliographystyle{plain}
\bibliography{references}

\begin{thebibliography}{10}

\bibitem{Alp65}
J.~L. Alperin.
\newblock Large abelian subgroups of $p$-groups.
\newblock {\em Transactions of the American Mathematical Society}, 117:10--20,
  1965.

\bibitem{Bae38}
Reinhold Baer.
\newblock Groups with abelian central quotient group.
\newblock {\em Transactions of the American Mathematical Society},
  44(3):357--386, 1938.

\bibitem{Bar16}
Alexander Barvinok.
\newblock {\em The Independence Polynomial}, pages 181--227.
\newblock Springer International Publishing, Cham, 2016.

\bibitem{BCG+21}
Xiaohui Bei, Shiteng Chen, Ji~Guan, Youming Qiao, and Xiaoming Sun.
\newblock From independent sets and vertex colorings to isotropic spaces and
  isotropic decompositions: Another bridge between graphs and alternating
  matrix spaces.
\newblock {\em {SIAM} J. Comput.}, 50(3):924--971, 2021.

\bibitem{BGH87}
Joe Buhler, Ranee Gupta, and Joe Harris.
\newblock Isotropic subspaces for skewforms and maximal abelian subgroups of
  $p$-groups.
\newblock {\em Journal of Algebra}, 108(1):269--279, 1987.

\bibitem{CS07}
Maria Chudnovsky and Paul Seymour.
\newblock The roots of the independence polynomial of a clawfree graph.
\newblock {\em Journal of Combinatorial Theory, Series B}, 97(3):350--357,
  2007.

\bibitem{GH83}
Ivan Gutman and Frank Harary.
\newblock Generalizations of the matching polynomial.
\newblock {\em Utilitas Mathematica}, 24(1):97--106, 1983.

\bibitem{HQ21}
Xiaoyu He and Youming Qiao.
\newblock On the Baer--Lov{\'a}sz--Tutte construction of groups from graphs:
  isomorphism types and homomorphism notions.
\newblock {\em European Journal of Combinatorics}, 98:103404, 2021.

\bibitem{LM05}
Vadim~E Levit and Eugen Mandrescu.
\newblock The independence polynomial of a graph -- a survey.
\newblock In {\em Proceedings of the 1st International Conference on Algebraic
  Informatics}, volume 233254, pages 231--252. Aristotle Univ. Thessaloniki
  Thessaloniki, 2005.

\bibitem{LQ20}
Yinan Li and Youming Qiao.
\newblock Group-theoretic generalisations of vertex and edge connectivities.
\newblock {\em Proceedings of the American Mathematical Society},
  148(11):4679--4693, 2020.

\bibitem{LQWWZ}
Yinan Li, Youming Qiao, Avi Wigderson, Yuval Wigderson, and Chuanqi Zhang.
\newblock Connections between graphs and matrix spaces.
\newblock {\em Israel Journal of Mathematics}, 256(2):513--580, 2023.

\bibitem{Lov79}
L{\'{a}}szl{\'{o}} Lov{\'{a}}sz.
\newblock On determinants, matchings, and random algorithms.
\newblock In {\em {FCT}}, pages 565--574, 1979.

\bibitem{Ols78}
A.~Yu Ol'shanskii.
\newblock The number of generators and orders of abelian subgroups of finite
  $p$-groups.
\newblock {\em Mathematical notes of the Academy of Sciences of the USSR},
  23(3):183--185, 1978.

\bibitem{Qia20_extremal}
Youming Qiao.
\newblock Tur\'an and {Ramsey} problems for alternating multilinear maps.
\newblock {\em Discrete Analysis}, (12):22 pp., 2023.

\bibitem{Ros22}
Tobias Rossmann.
\newblock Enumerating conjugacy classes of graphical groups over finite fields.
\newblock {\em Bulletin of the London Mathematical Society}, 54(5):1923--1943,
  2022.
  

\bibitem{Tut47}
W.~T. Tutte.
\newblock The factorization of linear graphs.
\newblock {\em Journal of the London Mathematical Society}, s1-22(2):107--111,
  1947.
  
\bibitem{Wei06}
D.~Weitz,
\newblock Counting independent sets up to the tree threshold,
\newblock in {\em Proceedings of the 38th Annual ACM Symposium on Theory of Computing (STOC 2006)}, pp.~140--149, 2006.

\bibitem{Wil09b}
James~B. Wilson.
\newblock Finding central decompositions of $p$-groups.
\newblock {\em Journal of Group Theory}, 12(6):813--830, 2009.

\end{thebibliography}

\end{document}